\documentclass[12pt,reqno]{amsart}

\usepackage{fullpage}
\usepackage{amsmath}
\usepackage{amsthm}
\usepackage{amssymb}
\usepackage{amsrefs}
\usepackage{bm, bbm, mathrsfs}
\usepackage{graphicx}

\newcommand{\RR}{\ensuremath{\mathbb{R}}}
\newcommand{\CC}{\ensuremath{\mathbb{C}}}

\newcommand{\me}{\ensuremath{\mathrm{e}}}
\newcommand{\pd}{\ensuremath{\partial}}

\newcommand{\spm}{{\ensuremath{{\scriptscriptstyle\pm}}}}
\renewcommand{\sp}{{\ensuremath{{\scriptscriptstyle +}}}}
\newcommand{\sm}{{\ensuremath{{\scriptscriptstyle -}}}}

\newcommand{\beq}{\begin{equation}}
\newcommand{\eeq}{\end{equation}}
\newcommand{\beqn}{\begin{equation*}}
\newcommand{\eeqn}{\end{equation*}}
\newcommand{\ben}{\begin{enumerate}}
\newcommand{\een}{\end{enumerate}}

\newcommand{\lam}{\ensuremath{\lambda}}

\newcommand{\dif}{\ensuremath{\mathrm{d}}}

\newcommand{\mi}{\ensuremath{\mathrm{i}}}
\renewcommand{\vec}[1]{\ensuremath{\bm{#1}}}

\newcommand{\norm}[2]{\ensuremath{\|#1\|_{#2}}}

\newcommand{\jump}[1]{\ensuremath{[#1]}}

\DeclareMathOperator{\dv}{div}
\DeclareMathOperator{\grad}{grad}
\DeclareMathOperator{\sgn}{sgn}
\DeclareMathOperator{\re}{Re}
\renewcommand{\Re}{\re}
\DeclareMathOperator{\im}{Im}
\renewcommand{\Im}{\im}
\DeclareMathOperator{\sech}{sech}

\newcommand{\tr}{\ensuremath{\mathrm{t}}}

\theoremstyle{plain}
\newtheorem{theorem}{Theorem}
\newtheorem{lemma}[theorem]{Lemma}

\newtheorem{proposition}[theorem]{Proposition}

\theoremstyle{definition}
\newtheorem{definition}[theorem]{Definition}
\newtheorem{example}[theorem]{Example}

\theoremstyle{remark}
\newtheorem{remark}[theorem]{Remark}

\newtheorem{note}[theorem]{Remark}
\numberwithin{equation}{section}
\begin{document}
\title[Refined stability]{Computing the refined stability condition}
\author{Nicholas B. Anderson}
\address{Department of Mathematics\\ University of Wyoming\\ Laramie, WY 82071-3036}
\email{nanders5@uwyo.edu}
\author{Allison M. Lindgren}
\address{Department of Mathematics\\ University of Wyoming\\ Laramie, WY 82071-3036}
\email{alindgr1@uwyo.edu}
\author{Gregory D. Lyng}
\address{Department of Mathematics\\ University of Wyoming\\ Laramie, WY 82071-3036}
\email{glyng@uwyo.edu}

\date{Last Updated: \today}
\begin{abstract}
The classical (inviscid) stability analysis of shock waves is based on the Lopatinski\u\i\ determinant, $\Delta$---a function of frequencies whose zeros determine the stability of the underlying shock. A careful analysis of $\Delta$ shows that in some cases the stable and unstable regions of parameter space are separated by an open set of parameters. Zumbrun and Serre [\emph{Indiana Univ. Math. J.}, \textbf{48} (1999) 937--992] have shown that, by taking account of viscous effects not present in the definition of $\Delta$, it is possible to determine the precise location in the open, neutral set of parameter space at which stability is lost. In particular, they show that the transition to instability under suitably localized perturbations is determined by an ``effective viscosity'' coefficient. Here, in the simplest possible setting, we propose and implement two new approaches toward the practical computation of this coefficient. Moreover, in a special case, we derive an exact solution of the relevant differential equations.
\end{abstract}
\maketitle

\section{Introduction}
\subsection{Inviscid stability}
Consider a planar shock solution,
\beq\label{eq:shock_soln}
u(\vec{x},t)=\begin{cases} u_\sp\,, & x_1>st \\ u_\sm\,, & x_1<st \end{cases}\,,
\eeq
of a hyperbolic system of conservation laws in $d$ space dimensions:
\beq\label{eq:general_claw}
\pd_t u +\sum_{j=1}^d \pd_{x_j}f^j(u)=0\,.
\eeq
Here, the unknown $u$ belongs to $\mathscr{U}\subset\RR^n$, the state space, and is a function of space $\vec{x}=(x_1,\ldots,x_d)\in\RR^d$ and time $t\in\RR$. The fluxes $f^j$ are $\RR^n$-valued functions on $\mathscr{U}$. In \eqref{eq:shock_soln}, $u_\spm$ are constant states; they are related to the shock speed $s$ by the Rankine--Hugoniot condition
\[
s\jump{u}=\jump{f^1(u)}\,,
\]
where---here and below---square brackets indicate the jump. That is, for any function $h$ of the state $u$, $\jump{h(u)}:=h(u_\sp)-h(u_\sm)$. The linear stability analysis of such solutions is now classical; it originates in the studies of pioneers D{\cprime}yakov~\cite{D_JETP54} and Erpenbeck~\cite{E_PF62}. We recall that the the centerpiece of the analysis is the \emph{Lopatinski\u{\i} determinant}, 
\[
\Delta:\{\lambda\in\CC\,:\re\lambda\geq 0\}\times\RR^{d-1}\to\CC\,,
\]
a function of frequencies $(\lambda,\tilde\xi)$, where $\lambda=\gamma+\mi\tau\in\CC$ is dual to time and $\tilde\xi\in\RR^{d-1}$ is dual to the transverse spatial directions. Zeros of $\Delta$ with $\gamma=\re\lambda>0$ correspond to perturbations which grow exponentially in time, and, evidently, nonvanishing of $\Delta$ on $\{\gamma>0\}$ is necessary for (linearized, inviscid) stability of the shock\footnote{Indeed, $\Delta$ is homogenous degree one. Thus, any unstable zero generates instabilities of all orders. That is, any such instability is of Hadamard type; these instabilities are so violent that these waves will never be seen in practice.}. 

A particularly important example occurs when \eqref{eq:general_claw} is the Euler equations of gas dynamics and the solution \eqref{eq:shock_soln} represents a planar gas-dynamical shock wave. In this example, the vector $u$ would have as its components the mass density, the momentum densities in each of the $d$ spatial directions, and the total energy density. In this case, an analysis of $\Delta$, see \cite{B-GS_book} or the appendix of \cite{Z_handbook04}, yields clean, explicit stability criteria in terms of the basic properties of the shock and the gas\footnote{The stability criteria are, naturally, formulated in terms of the end states $u_\spm$ of the shock and in terms of the equation of state.} in question. However, Majda \cite{M_book} has observed that in the setting of gas dynamics it sometimes happens that $\Delta$ vanishes on the boundary $\{\gamma=0\}$ but not for $\gamma>0$, and in this case we say that the shock is \emph{neutrally} or \emph{weakly} stable. Moreover, this neutral stability can persist on an open set in parameter space. Physically, the presence of neutral zeros of $\Delta$ corresponds to surface or boundary waves, and these zeros are associated with a loss of smoothness of the perturbed shock front. Further complicating matters, Barmin \& Egorushkin \cite{BE_AM92} have pointed out that experimentally observed instabilities of gas-dynamical shocks sometimes occur within the region of weak stability. That is, instabilities are sometimes observed  \emph{in the interior} of the weakly stable regime. They postulated that the inclusion of nonlinear effects and/or other neglected physical effects would be required for the theory to capture the phenomena revealed by these experiments.

\subsection{Viscous stability}
By considering the viscous regularization of \eqref{eq:general_claw},
\beq\label{eq:general_visc}
\pd_t u +\sum_{j=1}^d\pd_{x_j} f^j(u)_{x_j}=\nu\sum_{j,k=1}^d \pd_{x_j}(B^{jk}(u)\pd_{x_k}u)\,,
\eeq
Zumbrun \& Serre \cite{ZS_IUMJ99} have shown that the precise location of the transition to instability is entirely determined by viscous effects which are neglected in the construction of $\Delta$. See also the more recent paper of Benzoni-Gavage, Serre, \& Zumbrun \cite{B-GSZ_ZAA08}; this latter paper forms the foundation of the calculations we present here\footnote{Also, the original, one-dimensional ($d=1$) derivation appears in \cite{B-GSZ_SIAMJMA01}.}.  As a point of reference, we note that in the aforementioned example of gas dynamics, equation \eqref{eq:general_visc} would correspond to the Navier--Stokes equations of compressible gas dynamics; the second-order terms on the right-hand side of \eqref{eq:general_visc} model the effects of viscosity and heat conductivity. 

We associate to the solution \eqref{eq:shock_soln} of the inviscid equations \eqref{eq:general_claw} a planar viscous profile for \eqref{eq:general_visc}. The viscous profile is a solution of \eqref{eq:general_visc} of the form
\beq\label{eq:vp}
u(\vec{x},t)=\bar u\left(\frac{x_1-st}{\nu}\right)\,,\quad \lim_{z\to\pm\infty}\bar u(z)=u_\spm\,.
\eeq
Zumbrun \& Serre's derivation is based on a low-frequency analysis of the Evans function, $D(\lam,\tilde\xi)$, associated with \eqref{eq:vp}. Analogous to the Lopatinski\u\i\ determinant, the Evans function is a spectral determinant; its zero set carries stability information for the planar viscous shock wave. In particular, the refined stability condition they derive is given in terms of the sign of the real part of a coefficient called $\beta$; see \eqref{eq:beta} below. That is, 
\[
\sgn\Re\beta>0 
\]
is a necessary condition for weak viscous stability.
Our principal aim here is to explore, in a simplified setting, the effective, practical computation of $\beta$. As pointed out by Benzoni-Gavage, Serre, \& Zumbrun \cite{B-GSZ_ZAA08}, the fundamental challenge for computing $\beta$ is the numerical approximation of the function $\tilde y$, the solution of an appropriate differential equation; see \eqref{eq:y} below. The determination of $\Re\sgn\beta$ has been identified as an important open problem for physical systems due its possible role as a signal for the onset of complex behavior \cite{Z_kochel}; see also Zumbrun's more recent work on the role of the refined stability condition in the development of cellular instabilities for shock waves \cite{Z_PD10}.
\subsection{Plan}
In \S\ref{sec:prelim} we describe the framework that we use as a testbed for computing $\tilde y$ (and therefore $\beta$). In particular, we restrict our attention to the case of a scalar conservation law with viscosity in two space dimensions. In \S\ref{sec:evans}, for completeness and to make the exposition here mostly self contained, we recapitulate the derivation of $\beta$ and related Evans-function analysis from \cite{B-GSZ_ZAA08}; notably, because we have restricted ourselves to the scalar case, our derivation is substantially streamlined from the general calculation presented there. In \S\ref{sec:compute} we make the principal contribution of this paper. We propose two distinct approaches toward computing $\tilde y$, and we implement these approaches in several concrete example problems. In a special case, we are able to find an exact solution. We use this exact solution to validate our numerical approximations. Finally, in \S\ref{sec:conclude} we actually compute $\beta$ and discuss the steps that will be required to compute $\beta$ using these or similar techniques in the interesting (and physically relevant!) system case. These calculations represent, to the best of our knowledge, the first ever calculations of the refined stability condition in any setting.

\section{Preliminaries}\label{sec:prelim}
\subsection{Model}
We consider the simplest possible scenario of interest, a single conservation law with viscosity in two space dimensions:
\beq
\label{eq:law}
\pd_t u +\pd_{x_1}f^1(u)+\pd_{x_2}f^2(u) = \nu(\pd_{x_1}^2u+\pd_{x_2}^2u)\,.
\eeq
Here, $u$ is real valued, $\vec{x}=(x_1,x_2)\in\RR^2$, and $t\in\RR$ represents time. 
%
We shall write $a^j(u):=\frac{\dif f^j}{\dif u}(u)$. The parameter $\nu>0$ is the viscosity; for simplicity we shall take $\nu=1$. Equation \eqref{eq:law} is the viscous correction of the hyperbolic conservation law 
\beq\label{eq:hcl}
\pd_tu+\pd_{x_1}f^1(u)+\pd_{x_2}f^2(u)=0\,.
\eeq
Evidently, \eqref{eq:hcl} is hyperbolic since the $1\times 1$ matrix $\mathscr{A}(u;\vec{\omega})=a^1(u)\omega_1+a^2(u)\omega_2$ is always real for $(u,\vec{\omega})\in\RR\times \RR^2$. As is well known, such systems support discontinuous solutions (shocks), and we consider here the simplest possible shock solution, a planar shock wave connecting constant states. That is, we suppose as above that 
\beq\label{eq:shock}
\check u(\vec{x},t)=\begin{cases} u_\sp\,, & x_1>st \\ u_\sm\,, & x_1<st \end{cases}
\eeq
is a solution of \eqref{eq:hcl}. We recall that, in order for $\check u$ to be a weak solution of \eqref{eq:hcl},
$u_\spm$ and $s$ must satisfy the Rankine--Hugoniot condition
\beq\label{eq:rh}
s\jump{u}=\jump{f^1(u)}\,.
\eeq
\subsection{Lopatinski\u\i\ determinant}
Applying the classical stability analysis to $\check u$ amounts to the algebraic construction of the Lopatinski\u\i\ determinant $\Delta$.   
In general, i.e., in the case of \eqref{eq:general_claw}, $\Delta$ is constructed from a jump term 
\[
\lambda\jump{u}+\mi\sum_{j=2}^d\xi_j\jump{f^j(u)}
\]
together with bases for the stable ($-$) and unstable ($+$) subspaces of the matrices
\beq\label{eq:calA}
\mathcal{A}_\spm(\lambda,\tilde\xi)=\left(\lambda I+\mi\sum_{j=2}^d\xi_j\dif f^j(u_\spm)\right)(\dif f^1(u_\spm))^{-1}\,.
\eeq
For a Lax $p$-shock, the dimensions work out precisely since in that case
\[
\dim E^s(\mathcal{A}_\sm)=p-1\,,
\quad
\dim E^u(\mathcal{A}_\sp)=n-p\,.
\]
However, in our setting ($n=1$, $p=1$), the Lopatinski\u\i\ determinant consists only of the jump term, and 
we can write it down explicitly as
\beq\label{eq:lop}
\Delta(\lambda,\xi)=\lambda\jump{u}+\mi\xi\jump{f^2(u)}\,.
\eeq
First, observe that provided $\jump{u}\neq 0$, there are no zeros of $\Delta$ with $\Re\lambda>0$. 
Evidently, given a pair of states $u_\spm$ and a flux function $f^2$, it is always possible to find a purely imaginary value of $\lambda$ for which $\Delta$ vanishes. Namely, one can simply take 
\[
\lambda=-\mi\xi\frac{\jump{f^2(u)}}{\jump{u}}\,.
\]
The fact that there is a whole line of neutral zeros parametrized by $\xi$ is a manifestation of the homogeneity of $\Delta$. 
\begin{remark}
The above calculation shows that planar shocks for \eqref{eq:hcl} are \emph{always} weakly stable; see  the discussion in \cite{B-GS_book}. Thus, in the current setting, the goal of computing of $\beta$ could be regarded as artificial because it does not---in this case---serve to locate the transition point in parameter space separating the (strongly) stable and unstable regions. On the other hand, the simple form of $\Delta$ in \eqref{eq:lop} and the resulting abundance of weakly stable shocks makes this setting ideal for testing various approaches to the computation of $\beta$. The extension of these ideas to a physically relevant case with $n\geq 2$ is part of our ongoing work \cite{ABLLMSX}; in \S\ref{ssec:discuss} we indicate some features of our computations which might be useful in the setting of systems. 
\end{remark}

\subsection{Traveling-wave solutions}
We now turn to the equation with viscosity.
Recall, we have set $\nu=1$, and we seek solutions of \eqref{eq:law} of the form
\beq\label{eq:twa}
u(\vec{x},t)=\bar u(x-st)\,,\quad\lim_{z\to\pm\infty}\bar u(z)=u_\spm\,.
\eeq
We write $z=x-st$, and we note that the traveling-wave ansatz \eqref{eq:twa} reduces the partial differential equation \eqref{eq:law} to 
\beq\label{eq:tw1}
-s\frac{\dif \bar u}{\dif z}+\frac{\dif}{\dif z}f^1(\bar u) = \frac{\dif^2 \bar u}{\dif z^2}\,. 
\eeq
Integrating \eqref{eq:tw1} once, we find 
\beq\label{eq:tw2}
-s(\bar u-u_\sm)+f^1(\bar u)-f^1(u_\sm)=\bar u'\,,\quad '=\frac{\dif}{\dif z}\,.
\eeq 
A necessary condition for the existence of a traveling-wave profile is that both states $u_\spm$ be equilibria of the equation \eqref{eq:tw2}. Evidently, $u_\sm$ is an equilibrium. 
From the requirement that $u_\sp$ be an equilibrium we recover the Rankine--Hugoniot condition \eqref{eq:rh}:
\[
\jump{f^1(u)}=s\jump{u}\,.
\]

\begin{example}[Burgers flux]
In the case that $f^1(u)=u^2/2$, the Rankine--Hugoniot condition \eqref{eq:rh} reduces to 
\[
s=\frac{1}{2}(u_\sp+u_\sm)\,,
\]
so the shock speed $s$ is simply the average of the values of the end states. We take
\[
u_\sm=1\,,\quad u_\sp=-1
\]
so that $s=0$. Equation \eqref{eq:tw2} reduces in this case to 
\beq\label{eq:tw-burgers}
\bar u'=\frac{1}{2}(\bar u^2-1)\,,
\eeq
and a straightforward and well-known calculation shows that 
\(
\bar u(z)=-\tanh(z/2).
\)
We shall use this exact solution of the profile equation \eqref{eq:tw-burgers} to validate our numerical calculations below. See also Example \ref{ex:exact} below.
\end{example}

We note that the two Lax shock conditions
\beq\label{eq:lax}
a^1(u_\sp)-s<0\,,\quad a^1(u_\sm)-s>0
\eeq 
guarantee that the equilibrium at $u_\sp$ is a stable node and that the equilibrium at $u_\sm$ is an unstable node. Thus, since the phase space for the nonlinear differential equation is one dimensional, provided that there are no equilibria between $u_\sp$ and $u_\sm$, existence of a monotone profile $\bar u$ is immediate. Moreover, the profile will approach its limiting values $u_\spm$ exponentially fast.

\begin{note}
By a modification of the flux $f^1$ if necessary, we may assume that $s=0$. We make this assumption throughout the remainder of the paper. Thus, from this point forward the profile $\bar u$ is a standing wave, a function of $x_1$ alone. 
\end{note}

\subsection{Linearization}
The first step in the stability analysis is to linearize \eqref{eq:law} about the standing-wave solution $\bar u$. We obtain
\beq\label{eq:linearv}
\pd_tv+\pd_{x_1}(a^1(\bar u) v)+\pd_{x_2}(a^2(\bar u)v)= \pd_{x_1}^2v+\pd_{x_2}^2v\,.
\eeq
Here, $v=v(\vec{x},t)$ denotes the perturbation, and equation \eqref{eq:linearv} describes the approximate (linear) evolution of the perturbation $v$. In particular, if the linearized equation supports solutions $v$ which grow in time, the solution $\bar u$ will be unstable. The linear equation \eqref{eq:linearv} has variable coefficients (since $\bar u$ is nonconstant), but these coefficients are functions of $x_1$ alone. Thus, we may take the Laplace transform in $t$ (dual variable $\lambda\in\CC$) and the Fourier transform in $x_2$ (dual variable $\xi\in\RR$). The transformed equation, with $w$ denoting the transformed perturbation, takes the form
\beq\label{eq:fl}
\lam w+(a^1(\bar u) w)'+\mi\xi a^2(\bar u)w= w''-\xi^2w\,.
\eeq
Here, $'$ denotes differentiation with respect to $x_1$. Indeed, from this point forward, we omit the superfluous subscript $1$ on $x_1$; it is the only surviving spatial variable. Thus, $x$ will denote the spatial coordinate normal to the unperturbed shock front. 
We think of \eqref{eq:fl} as a family of eigenvalue problems for the collection of linear operators $\mathscr{L}=\mathscr{L}(\xi)$ parametrized by $\xi\in\RR$ and defined by 
\beq
\mathscr{L}(\xi)w:=\big(w'-a^1(\bar u)w\big)'-\mi\xi a^2(\bar u)w-\xi^2w\,.
\eeq
We sometimes write $\mathscr{L}(\xi)w=\lambda w$ as a convenient shorthand for \eqref{eq:fl}.
Solutions of $\mathscr{L}(\xi)w=\lambda w$ which decay at $\pm\infty$ with $\Re\lambda>0$ correspond to perturbations which grow in time since $v$ can be recovered from $w$ via
\[
v(\vec{x},t)=\me^{\lambda t}\me^{\mi\xi x_2}w(x)\,.
\]
Evidently, since $\bar u(x)\to u_\spm$ as $x\to\pm\infty$, there are a pair of related, limiting operators 
\beq
\mathscr{L}_\spm(\xi)w:=w''-a^1(u_\spm)w'-\mi\xi a^2(u_\spm)w-\xi^2w\,.
\eeq
Notably, for every $\xi\in\RR$ the operators $\mathscr{L}_\spm(\xi)$ are constant-coefficient operators. Thus, they may be analyzed quite completely. This feature is an essential ingredient of the analysis. We also note for future reference that 
\[
\mathscr{L}(0)w=w''-(a^1(\bar u)w)'\,,
\]
and, evidently, from \eqref{eq:tw1} with $s=0$, we find  
\beq\label{eq:ubar}
\mathscr{L}(0)\bar u' = 0\,.
\eeq
Thus, $\bar u'$ is a decaying solution of \eqref{eq:fl} corresponding to $(\lambda,\xi)=(0,0)$.

\section{Evans function}\label{sec:evans}
\subsection{Evans function and low-frequency limit}\label{ssec:evanslf}
The Evans function is easily constructed in this case. We content ourselves with a mere outline of the procedure here. For more details in a setting which includes ours as a special case, see \cites{ZS_IUMJ99,B-GSZ_ZAA08}\,. First, we rewrite the eigenvalue problem $(\mathscr{L}(\xi)-\lambda) w=0$ as a first-order system of differential equations:
\beq\label{eq:eode}
W'=\mathbb{A}(x;\lambda,\xi)W\,.
\eeq
As above, we use $'$ to denote differentiation with respect to $x$, and we have written $W=(w,w')^\tr$. The coefficient matrix $\mathbb{A}(x;\lambda,\xi)$ is given by 
\beq
\mathbb{A}(x;\lambda,\xi)=
\begin{pmatrix}
0 & 1 \\
\lambda+\mi\xi a^2(\bar u)+\xi^2+a^1(\bar u)' & a^1(\bar u)
\end{pmatrix}\,.
\eeq
Corresponding to the constant coefficient operators $\mathscr{L}_\spm(\xi)$, there are constant coefficient first-order systems $\mathcal{W}'=\mathbb{A}_\spm(\lambda,\xi)\mathcal{W}$ with 
\beq\label{eq:apm}
\mathbb{A}_\spm(\lambda,\xi)=
\begin{pmatrix}
0 & 1 \\
\lambda+\mi\xi a^2(u_\spm)+\xi^2 & a^1(u_\spm)
\end{pmatrix}\,.
\eeq
The eigenvalues $\mu(\lambda,\xi)$ of $\mathbb{A}_\spm(\lambda,\xi)$ are roots of 
\beq\label{eq:cpoly}
\mu^2-\mu a^1(u_\spm)-\mi\xi a^2(u_\spm)-\xi^2-\lambda =0\,.
\eeq
Observe that as long as $\Re\lambda>0$ there can be no imaginary root of \eqref{eq:cpoly}. To see this, simply observe that if $\mu=\mi\eta$ for $\eta\in\RR$, then the left-hand side of \eqref{eq:cpoly} has as its real part the expression  
\[
-(\eta^2+\xi^2)-\lambda\,,
\]
which is clearly negative provided that $\Re\lambda>0$. Thus, the matrices $\mathbb{A}_\spm$ have no center subspace on $H=\{\Re\lambda>0\}\times\RR$, and one can therefore determine the dimensions of the stable and unstable subspaces of $\mathbb{A}_\spm$ by taking $\xi=0$. Consideration of \eqref{eq:cpoly} then shows that 
the stable and unstable subspaces of the two matrices each have dimension one on $H$. Of particular interest are the stable subspace $E^s(\mathbb{A}_\sp(\lambda,\xi))$ and the unstable subspace $E^u(\mathbb{A}_\sm(\lambda,\xi))$. The next lemma shows that there are solutions of the variable-coefficient problem \eqref{eq:eode} which asymptotically tend to zero (forward or backward in $x$) in the directions of these invariant subspaces.

\begin{lemma}
Fix a base point $P_\circ=(\lambda_\circ,\xi_\circ)\in\{\Re\lambda>0\}\times\RR$.
There exist solutions $W_\spm=W_\spm(x;\lambda,\xi)$ of \eqref{eq:eode} such that 
\[
W_\spm\to 0 \quad\text{as}\quad x\to\pm\infty.
\]
Local to $P_\circ$, these solutions are holomorphic in $\lambda$ and real analytic in $\xi$.
\end{lemma}
\begin{proof}
The proof is a consequence of the preceding constant-coefficient analysis and the conjugation lemma of \cite{MZ_MAMS05}. In particular, the conjugation lemma asserts the existence of a well-behaved invertible change of coordinates which maps solutions $\mathcal{W}$ of the constant-coefficient problem to solutions $W$ of the variable-coefficient one. 
See \cite{ZS_IUMJ99} for details.
\end{proof}

\begin{definition}[Evans function]
The (local) \emph{Evans function} $D$ is defined by 
\beq\label{eq:evans}
D(\lambda,\xi)=\det(W_\sp(0;\lambda,\xi),W_\sm(0;\lambda,\xi))\,.
\eeq
\end{definition}
Our principal interest is in the polar Evans function defined for radial coordinate $\rho\in(0,\infty)$ by
\[
D(\rho;\lambda,\xi):=D(\rho\lambda,\rho\xi)\,.
\]
For ease of notation, we sometimes suppress the dependence on $\lambda$ and $\xi$ and simply denote the Evans function by $D(\rho)$. 

\begin{remark}
The key point is to understand what happens when $\rho\to 0^+$ and when $\Re\lambda=0$. 
In general, a delicate issue that arises on the boundary of $H$ are the \emph{glancing} and \emph{variable multiplicity} sets $\mathscr{G}$ and $\mathscr{V}$. However, when $n=1$, a direct computation shows that 
\[
\mathscr{G}=\emptyset\,,\quad \mathscr{V}=\emptyset. 
\]
Thus, we need not concern ourselves with these sets here. In the general case, the presence of glancing points, for example, prohibits the smooth extension of the stable and unstable subspaces of $\mathbb{A}_\spm$ to $\Re\lambda=0$, and these points---a measure zero set---must be excised from the boundary of $H$ in advance of the low-frequency analysis of $D$ that follows.
See \cite{B-GSZ_ZAA08} or \cite{Z_cime07} for more details.
\end{remark}

The next lemma paves the way for the fundamental low-frequency analysis of $D$ that is the cornerstone of the derivation of $\beta$. We omit the proof which is technical; details can be found in \cite{ZS_IUMJ99}.
\begin{lemma}[Low-frequency extension]\label{lem:lf}
Fix a base point $P_\circ=(\lambda_\circ,\xi_\circ)\in\{\Re\lambda\geq 0\}\times\RR$. The Evans function $D$ and its factors $W_\spm=(w_\spm,w_\spm')$ have a unique jointly analytic extension onto a neighborhood of $(\rho;\lambda,\xi)=(0;P_\circ)$. Moreover, the factors may be chosen so that
\[
w_\spm''-(a^1(\bar u)w_\spm)'=0\,,
\]
with $w_\sp(+\infty)=0$ and $w_\sm(-\infty)=0$.
\end{lemma}
With Lemma \ref{lem:lf} in hand, we are ready to outline the proof of the fundamental result of Zumbrun \& Serre \cite{ZS_IUMJ99}. Importantly, it links in a rigorous way the Lopatinski\u\i\ determinant $\Delta$ encoding the stability of the inviscid, ideal shock \eqref{eq:shock} and the Evans function associated with the corresponding viscous profile \eqref{eq:twa}. Roughly speaking, the result says that low-frequency/long-wave perturbations cannot distinguish between the inviscid, ideal shock $\check u$ and the viscous profile $\bar u$. 
\begin{proposition}[Zumbrun \& Serre, \cite{ZS_IUMJ99}]\label{prop:d1}
Fix $(\lambda,\xi)\in \overline{H}=\{\Re\lambda\geq0\}\times\RR$. Then, 
\beq
D(\lambda,\xi)=\Gamma\Delta(\lambda,\xi)+O(|\lambda,\xi|^2).
\eeq
Here, $\Gamma=\bar u'(0)$ is a constant measuring transversality of the connection in the traveling-wave problem\footnote{Transversality is automatic in the $n=1$ case we consider in this paper.}.
The quadratic error term is uniform for $(\lambda,\xi)$ in bounded subsets of $\overline{H}$. 
Equivalently, 
$D(0;\lambda,\xi)=0$ and 
\beq
\pd_\rho D(0;\lambda,\xi)=\Gamma\Delta(\lambda,\xi)\,.
\eeq
\end{proposition}

\begin{proof}
We outline the principal calculation. First, recalling the discussion surrounding equation \eqref{eq:ubar}, we make the standard normalization at $\rho=0$ that $w_\spm(0;\lambda,\xi)=\bar u'$. From this, the assertion that $D(0)=0$ is immediate. Second, observe that 
\beq\label{eq:leibniz}
\pd_\rho D (0) =  \det(\pd_\rho W_\sp, W_\sm)+\det(W_\sp ,\pd_\rho W_\sm)\,.
\eeq
But then, using again the normalization at $\rho = 0$, we see that the expression for the derivative of the Evans function in \eqref{eq:leibniz} may be rewritten as 
\beq
\pd_\rho D(0)=\det( \bar U', Y_\sm-Y_\sp)\,,
\eeq
where $Y_\spm=(y_\spm,y_\spm')^\tr:=\pd_\rho W_\spm$ and $\bar U'=(\bar u',\bar u'')^\tr$. 
In polar coordinates the eigenvalue equation takes the form 
\beq\label{eq:rho}
\rho\lam w+(a^1(\bar u) w)'+\mi\rho\xi a^2(\bar u)w= w''-\rho^2\xi^2w\,.
\eeq
Thus, differentiating with respect to $\rho$ and setting $\rho=0$, we find from \eqref{eq:rho} that 
\beq\label{eq:rho-zero}
\lam\bar u'+(a^1(\bar u)y_\spm)'+\mi\xi a^2(\bar u)\bar u'=y_\spm''\,.
\eeq
We rewrite \eqref{eq:rho-zero} as 
\beq\label{eq:rho-perfect}
\big(y_\spm'-(a^1(\bar u)y_\spm\big)'=\big(\lam\bar u +\mi\xi f^2(\bar u)\big)'\,,
\eeq
to express both sides of the equation as perfect derivatives. We integrate the $y_\sp$ equation in  \eqref{eq:rho-perfect} from $+\infty$ to $x$ (Recall, $y_\sp(+\infty)=0, y_\sp'(+\infty)=0$):
\beq\label{eq:yp}
y_\sp'-a^1(\bar u)y_\sp=\lam(\bar u-u_\sp)+\mi\xi (f^2(\bar u)-f^2(u_\sp))\,.
\eeq
Similarly, we integrate the $y_\sm$ equation from $-\infty$ to $x$: 
\beq\label{eq:ym}
y_\sm'-a^1(\bar u)y_\sm=\lam(\bar u-u_\sm)+\mi\xi(f^2(\bar u)-f^2(u_\sm))\,.
\eeq
Combining the results of \eqref{eq:yp} and \eqref{eq:ym}, we find 
that the components of $Y=Y_\sm-Y_\sp$ satisfy the equation 
\beq\label{eq:yzs}
y'-a^1(\bar u )y=\lambda\jump{u}+\mi\xi\jump{f^2(u)}\,.
\eeq
But, as we noted above in \eqref{eq:ubar}, $\bar u'$ satisfies $\mathscr{L}(0)\bar u'=0$, or 
\beq\label{eq:baru-row}
\bar u''-a^1(\bar u)\bar u'=0\,.
\eeq
Equations \eqref{eq:yzs} and \eqref{eq:baru-row} show that the row operation of adding $a^1(\bar u)$ times the first row to the second row simplifies the Evans determinant as follows
\beq
\pd_\rho D(0)=\det\begin{pmatrix} \bar u' & y \\ \bar u'' & y' \end{pmatrix}
=\det\begin{pmatrix} \bar u' & * \\ 0 & \Delta\end{pmatrix}=\Gamma\Delta\,.
\eeq
This completes the outline of the proof. 
\end{proof}

\subsection{The building blocks of $\beta$}\label{ssec:beta}

In this section we outline the derivation of the final ingredient of $\beta$; these calculations provide the framework for our computational approach. They depend on the calculations in \S\ref{ssec:evanslf}, particularly those in the proof of Proposition \ref{prop:d1}. The basic assumption is that $(\lambda_\circ,\xi_\circ)$ is a zero of $\Delta$ with $\gamma_\circ=\Re\lambda_\circ=0$. In the current setting, this means that $\lambda_\circ$ and $\xi_\circ$ are related via
\beq
\tau_\circ=\Im\lambda_\circ=-\xi_\circ\frac{\jump{f^2(u)}}{\jump{u}}\,.
\eeq

\begin{proposition}[Benzoni-Gavage, Serre, \& Zumbrun, \cite{B-GSZ_ZAA08}]\label{prop:d2}
If $(\lam_\circ,\xi_\circ)$ is a neutral zero of $\Delta$, then 
\beq
\pd_\rho^2D(0)=\Gamma\int_{-\infty}^\infty 2\big(\mi\tau_\circ+\mi\xi_\circ a^2(\bar u(\eta))\big)y(\eta)+2\xi_\circ^2\bar u'(\eta)\,\dif\eta\,.
\eeq
\end{proposition}
\begin{proof}
First, we note that the assumption $\Delta=0$ implies that (see \eqref{eq:yzs})
\[
y'-a^1(\bar u)y=0\,.
\]
But, this implies that,
for some constant $c_1$,
\[
y_\sm-y_\sp=c_1\bar u'\,.
\]
Therefore, we define 
\beq
\tilde w_\sm=w_\sm+\rho c_1w_\sm\,,\quad \tilde w_\sp=w_\sp\,,
\eeq
and we note that at $\rho=0$, we have
\[
\tilde w_\sp=\tilde w_\sm=\bar u'\,,\quad
\tilde y_\sm=y_\sm+c_1\bar u'\,,\quad
\tilde y_\sp=y_\sp\,.
\]
We shall use this new basis for our computations. That is, we now compute with the alternatively defined Evans function
\beq
D(\lambda,\xi)=\det(\tilde W_\sp(0;\lam,\xi), \tilde W_\sm(0;\lam,\xi)).
\eeq 
The advantage is that, with this new basis, 
\beq\label{eq:newbasis}
\tilde y_\sm=\tilde y_\sp\,.
\eeq
Then, as in Proposition \ref{prop:d1}, $D(0)=0$ and (now also) $\pd_\rho D(0)=0$.
We now examine the second derivative of $D$ with respect to $\rho$. We observe that, by the Leibniz rule, the second derivative of the Evans determinant may be expanded as
\beq\label{eq:leib2}
\pd_\rho^2 D = \det (\pd_\rho^2 \tilde W_\sp,\tilde W_\sm)+2\det(\pd_\rho \tilde W_\sp,\pd_\rho \tilde W_\sm)+\det(\tilde W_\sp,\pd_\rho^2\tilde W_\sm)\,.
\eeq
But, by \eqref{eq:newbasis}, it follows immediately that 
\begin{align*}
\det (\pd_\rho \tilde W_\sp,\pd_\rho \tilde W_\sm) & = \det\begin{pmatrix} \tilde y_\sp & \tilde y_\sm \\ \tilde y_\sp' & \tilde y_\sm'\end{pmatrix} =0\,.
\end{align*}
Thus, we may proceed in a fashion similarly as in Proposition \ref{prop:d1}. The two remaining determinants in \eqref{eq:leib2} may be combined so that
\beq
\pd_\rho^2 D(0)=\det (\bar U', \tilde Z_\sm-\tilde Z_\sp)
\eeq
where
\beq
\tilde Z_\spm=(\tilde z_\spm,\tilde z_\spm')^\tr:=\frac{\pd^2\tilde W_\spm}{\pd\rho^2}\,.
\eeq
We differentiate \eqref{eq:rho} twice with respect to $\rho$ and we set $\rho=0$; we find that $\tilde z_\spm$ satisfy the equation
\beq
2\lam_\circ \tilde y_\spm+(a^1(\bar u)\tilde z_\spm)'+2\mi\xi_\circ a^2(\bar u)\tilde y_\spm=\tilde z_\spm''-2\xi_\circ^2\bar u'\,,
\eeq
or, rearranging terms, 
\beq\label{eq:zpm}
\mathscr{L}(0)\tilde z_\spm=\big(\tilde z_\spm'-a^1(\bar u)\tilde z_\spm\big)'=2(\lambda_\circ+\mi\xi_\circ a^2(\bar u))\tilde y_\spm+2\xi_\circ^2\bar u'\,.
\eeq
But, since $\tilde y_\sp=\tilde y_\sm$, we subtract and integrate in \eqref{eq:zpm} to find that $\tilde z=\tilde z_\sm-\tilde z_\sp$ satisfies
\beq\label{eq:constant}
\tilde z'-a^1(\bar u)\tilde z=\mathcal{I}\,,
\eeq
where $\mathcal{I}$ is a constant. 
Next, by a row operation as in the proof of Proposition \ref{prop:d1}, we find that 
\beq
\pd_\rho^2D(0)=\det (\bar U', \tilde Z_\sm-\tilde Z_\sp)=\det\begin{pmatrix} \bar u' & * \\ 0 & \mathcal{I}\end{pmatrix}=\Gamma \mathcal{I}\,.
\eeq
It remains to identify the constant $\mathcal{I}$. We define $\mathscr{M}$ by $\mathscr{M}'=\mathscr{L}(0)$. That is, the action of $\mathscr{M}$ on a function $g$ is given by 
\[
\mathscr{M}g(x)=g'(x)-a^1(\bar u(x))g(x)\,.
\]
Thus, we may rewrite \eqref{eq:constant} as 
\beq\label{eq:constant2}
\mathcal{I}=\mathscr{M}\tilde z_\sm(0)-\mathscr{M}\tilde z_\sp(0)\,.
\eeq
But, by the fundamental theorem of calculus, 
\begin{align}
\int_{-\infty}^0\mathscr{L}(0)\tilde z_\sm(x)\,\dif x&=\mathscr{M}\tilde z_\sm(0)-\mathscr{M}\tilde z_\sm(-\infty)\,,\label{eq:ftc1}\\
\intertext{and}
\int_{0}^{+\infty}\mathscr{L}(0)\tilde z_\sp(x)\,\dif x&=\mathscr{M}\tilde z_\sp(+\infty)-\mathscr{M}\tilde z_\sp(0)\,.
\label{eq:ftc2}
\end{align}
Thus, using \eqref{eq:ftc1} and \eqref{eq:ftc2}, we are finally in a position to identify the constant $\mathcal{I}$. We see, starting with the expression in \eqref{eq:constant2}, that
\begin{align}
\mathcal{I} & = \mathscr{M}\tilde z_\sm(0)-\mathscr{M}\tilde z_\sp(0) \label{eq:a}\\
	& = \int_{-\infty}^0\mathscr{L}(0)\tilde z_\sm(x)\,\dif x + \int_0^{+\infty}\mathscr{L}(0)\tilde z_\sp(x)\,\dif x+\mathscr{M}\tilde z_\sm(-\infty)-\mathscr{M}\tilde z_\sp(+\infty) \label{eq:b}\\
	& = \int_{-\infty}^{+\infty}\mathscr{L}(0)\tilde z_\spm(x)\,\dif x+\mathscr{M}\tilde z_\sm(-\infty)-\mathscr{M}\tilde z_\sp(+\infty) \label{eq:c} \\
	& =\int_{-\infty}^{+\infty}2(\lambda_\circ+\mi\xi_\circ a^2(\bar u(x)))\tilde y(x)+2\xi_\circ^2\bar u'(x)\,\dif x\,.\label{eq:d}
\end{align}
We note that the equality in \eqref{eq:c} follows from the fact that $\mathscr{L}(0)\tilde z_\sp=\mathscr{L}(0)\tilde z_\sm$ (see \eqref{eq:zpm}), and the boundary terms vanish because $\tilde z_\spm$ tends to zero exponentially as $x\to\pm\infty$.
In conclusion, we have shown that  
\beq\label{eq:i}
\mathcal{I}=\int_{-\infty}^{+\infty}2(\lambda_\circ+\mi\xi_\circ a^2(\bar u(x)))\tilde y(x)+2\xi_\circ^2\bar u'(x)\,\dif x\,.
\eeq
\end{proof}
Propositions \ref{prop:d1} and \ref{prop:d2} provide the fundamental ingredients for the derivation and calculation of $\beta$.  The derivation is based on the relationship between the zero sets of $\Delta$ and $D$ expressed in Proposition \ref{prop:d1}. The main analytic tool is the implicit function theorem. We omit the details, which may be found in 
\cite{ZS_IUMJ99}, and simply note that $\beta$ is given by 
\beq\label{eq:beta}
\beta=\left(\frac{\pd^2 D}{\pd\rho^2}\right)\left(\frac{\pd^2 D}{\pd\rho\pd\lambda}\right)^{-1}\,.
\eeq
But, Proposition \ref{prop:d1} implies that 
\[
\frac{\pd^2D}{\pd\rho\pd\lam}=\Gamma\frac{\pd\Delta}{\pd\lam}\,,
\]
and---given the explicit form of $\Delta$ in \eqref{eq:lop}---we see immediately that 
\[
\frac{\pd \Delta}{\pd\lambda}=\jump{u}\,.
\]
Thus, from Proposition \ref{prop:d2}, we find that
\beq\label{eq:beta1}
\beta   =\frac{\Gamma\mathcal{I}}{\Gamma\Delta_\lambda}\,.
\eeq
From \eqref{eq:beta1} we see immediately that the computation of $\beta$ requires one to compute $\mathcal{I}$. Evidently, from \eqref{eq:i}, to compute $\mathcal{I}$ one needs to know both the profile $\bar u$ and the function $\tilde y$ solving the differential equation (see \eqref{eq:rho-perfect})
\beq\label{eq:y}
\big(\tilde y'-a^1(\bar u) \tilde y\big)'=\big(\lambda_\circ+\mi\xi_\circ a^2(\bar u(x))\big)\bar u'\,.
\eeq
In the next section we propose two methods for the practical approximation of $\tilde y$.

\begin{note}
 Briefly, the sign of $\Re\beta$ detects whether or not the zero level set of $D$ curls into the unstable half plane. Heuristically, it has a physical interpretation as an ``effective viscosity'' coefficient for transversely propagating deformations of the front. A detailed discussion of this point can be found in \cite{ZS_IUMJ99}. 
\end{note}

\section{Computing $\tilde y$}\label{sec:compute}

As pointed out in \cite{B-GSZ_ZAA08}, the principal task in finding $\beta$ is to compute $\tilde y$; the other elements required to compute $\beta$ are the building blocks of $\Delta$---in general these are the eigenvalues and eigenvectors\footnote{In the general case, one also uses these eigenvalues and eigenvectors to compute a boundary term $\mathcal{B}$; this term does not appear in our setting.} of $\mathcal{A}_\spm(\lambda,\tilde\xi)$---and the profile $\bar u$.
Benzoni-Gavage, Serre, \& Zumbrun \cite{B-GSZ_ZAA08} have proposed a two-step method for finding $\beta$. That is, the equation for $\tilde y$ is a linear equation whose coefficients depend on the profile $\bar u$. Thus, they propose to first solve the profile equation \eqref{eq:tw2}. For gas dynamics, techniques for doing this are well known; see, e.g., \cite{HLZ_ARMA09}. Then, they describe how to transform the equation for $\tilde y$ into one that fits into a standard numerical framework for approximating the solution of a two-point boundary value problem. Their method has never, to our knowledge, been implemented. Below, we propose two alternative approaches that, at least in the present context ($n=1$),  work well. We are currently investigating all three approaches in the case of gas dynamics for which $n>1$.

\subsection{Integrating factor}\label{ssec:if}

This method exploits the linear structure of \eqref{eq:y}. The equation for $\tilde y$ is  
\beq\label{eq:ytif}
\big(\tilde y'-a^1(\bar u) \tilde y\big)'=(\mi\tau_\circ+\mi\xi_\circ a^2(\bar u(x)))\bar u'\,.
\eeq
We integrate both sides of \eqref{eq:ytif}, and we obtain
\beq
\tilde y'-a^1(\bar u) \tilde y=\mi\tau_\circ \big(\bar u-u_\sm\big)+\mi\xi_\circ\big(f^2(\bar u)-f^2(u_\sm)\big)\,.
\eeq
We write $\tilde y$ in terms of its real and imaginary parts as 
\[
 \tilde y= w+\mi v\,.
\]
Then, the relevant system of equations becomes
\begin{subequations}\label{eq:gen-vw}
\begin{align}
 w' & = a^1(\bar u) w\,, \\
 v' & = a^1(\bar u) v +\tau_\circ(\bar u-u_\sm)+\xi_\circ(f^2(\bar u)-f^2(u_\sm))\,. 
\end{align}
\end{subequations}
We suppose that $\bar u$ is known, and we write 
\[
\mathcal{F}(x)=\tau_\circ(\bar u(x)-u_\sm)+\xi_\circ(f^2(\bar u(x))-f^2(u_\sm))\,.
\]
Now, suppose that $M$ is defined by
\beq\label{eq:int-factor}
M(x):=\exp\left(-\int_0^x a^1(\bar u(z))\,\dif z\right)\,.
\eeq
Then, a simple calculation shows that $M'(x)=-a^1(\bar u(x))M(x)$, whence the equation for $w$ can be rewritten as a perfect derivative $(M w)'=0$. Integrating from $0$ to $x$ and using $w(0)=A$, we find immediately that 
\beq\label{eq:w}
w(x)=A\exp\left(\int_0^xa^1(\bar u(z))\,\dif z\right)\,.
\eeq
Thus, $w$ is given explicitly in terms of the profile $\bar u$. Similarly, we apply the integrating factor $M$ to the equation for $v$. We write $v(0)=B$, and we see that 
\begin{align}
v(x)&=\me^{\int_0^x a^1(\bar u(z))\,\dif z}\left\{B+\int_0^x\me^{-\int_0^z a^1(\bar u(\eta))\,\dif\eta}\mathcal{F}(z)\,\dif z\right\} \nonumber \\
	&= \me^{\int_0^x a^1(\bar u(z))\,\dif z}\left\{B+\int_0^x\me^{-\int_0^z a^1(\bar u(\eta))\,\dif\eta}\Big(\tau_\circ(\bar u(z)-u_\sm)+\xi_\circ(f^2(\bar u(z))-f^2(u_\sm))\,\dif z\Big)\right\}\,. \label{eq:v}
\end{align}
Equations \eqref{eq:w} and \eqref{eq:v} show that we may write $\tilde y=w+\mi v$ explicitly in terms of the profile $\bar u$. Thus, rather than solve a differential equation for $\tilde y$, in this case we may simply approximate $\bar u$ (as described above, a well understood problem) and then use that approximation to compute the integrals in \eqref{eq:w} and \eqref{eq:v}. Indeed, as Example \ref{ex:exact} below shows, in some cases it is possible to compute these integrals exactly. 

\begin{remark}\label{rem:ab}
From \cite{B-GSZ_ZAA08}, we note that in the computation of $\beta$, the two apparently free parameters $A$ and $B$ above should be chosen so that $\tilde y$ satisfies an orthogonality condition at the origin. In the current setting, this reduces to the requirement that   
\[
A=0\,,\quad B=0\,.
\]
\end{remark}

\begin{example}[Exact solution]\label{ex:exact}
In the case of the Burgers flux $f^1(u)=u^2/2$, we have seen that if $u_\sp=-1$, $u_\sm=1$, then 
\(
s=0
\).
We have also seen that in this case the profile $\bar u$ is given by
\[
\bar u(x)=-\tanh(x/2)\,.
\]
Suppose that $f^2(u)=u^2$. We take $\tau_\circ=0$ and $\xi_\circ=1$ so that $\Delta(\mi\tau_\circ,\xi_\circ)=0$. Then, system of equations \eqref{eq:gen-vw} reduces to 
\beq\label{eq:exactvw}
w' = \bar u w\,, \quad 
v'  = \bar u v+\bar u^2-1\,.
\eeq
That is, we need to solve 
\[
w' = -\tanh(x/2) w\,, \quad 
v'  = -\tanh(x/2)v+\tanh^2(x/2)-1\,.
\]
By direct computation, with $A=0$ and $B=0$, we find immediately from \eqref{eq:w}, \eqref{eq:v} that 
\beq
w(x)=0\,,\quad v(x)=-x\sech^2(x/2)\,.
\eeq
This solution is plotted in Figure \ref{fig:compare}.
\end{example}

\begin{remark}
The formulae in \cite{B-GSZ_ZAA08} are derived in the case of strictly parabolic viscosity. That is, the matrices $B^{jk}$ in \eqref{eq:general_visc} are assumed to satisfy
\beq\label{eq:parabolicity}
\Re\sigma\left(\sum_{j,k=1}^d\xi_j\xi_k B^{jk}\big(\bar u(\cdot)\big)\right)>0 \quad\text{for all $\vec{\xi}\in\RR^d\setminus\{0\}$}\,.
\eeq
However, the important physical case of gas dynamics features only a partially parabolic or ``real'' viscosity. For example, consider the equations of isentropic gas dynamics
\begin{subequations}\label{eq:gas}
\begin{align}
\pd_t\rho +\nabla\cdot (\rho\vec{u}) &= 0\,, \label{eq:mass}\\
\pd_t(\rho\vec{u})  +\nabla\cdot(\rho \vec{u}\otimes\vec{u})+\nabla p &=\mu\Delta\vec{u}+(\mu+\eta)\grad\dv\vec{u} \,,\label{eq:momentum}
\end{align}
\end{subequations}
where $\mu$ and $\eta$ are the first (``dynamic'') and second viscosity coefficients (assumed here to be constant). The lack of second-order terms on the right-hand side of the conservation of mass equation \eqref{eq:mass} prevents the system \eqref{eq:gas} from satisfying \eqref{eq:parabolicity}; nonetheless, it is clear that the results of \cite{B-GSZ_ZAA08} extend in a natural way to systems with ``real'' or partially parabolic viscosity, such as the equations of gas dynamics. It is a rather tedious exercise to derive the equation for $\tilde y$ in that setting. 
Remarkably, our preliminary calculations for the system \eqref{eq:gas}  show that the equation for $\tilde y$ can be written as a linear \emph{diagonal} system \cite{ABLLMSX}; this suggests that the above approach based on integrating factors might be applicable to the corresponding calculation for \eqref{eq:gas}.
\end{remark}

\subsection{Coupled formulation}\label{ssec:auto}
\subsubsection{Description}
Rather than solve the problem in two steps (first $\bar u$ and then $\tilde y$), we consider now the problem of solving the coupled system for $\bar u$ and $\tilde y=w+\mi v$. Thus, we consider the autonomous system 
\begin{subequations}\label{eq:autosystem}
\begin{align}
\bar u' &=f^1(\bar u)-f^1(u_\sm)\,,\\
 w' & = a^1(\bar u) w\,, \\
 v' & = a^1(\bar u) v +\tau_\circ(\bar u-u_\sm)+\xi_\circ(f^2(\bar u)-f^2(u_\sm))\,. 
\end{align}
\end{subequations}
Our motivation for pursuing this approach is that, if the numerical methods involved do not take explicit advantage of the linear structure of the $\tilde y$ equation, for modestly sized systems it involves no extra work to solve the coupled system in a single step. On the other hand, the challenge in this case is to construct a suitable guess for the solver. However, once a suitable guess is found, the $\beta$ calculation can take advantage of continuation. For example, if $u_\sm$ is being moved along the Hugoniot curve, one can use the previously found solution as the initial guess for the next value of $u_\sm$; see Example \ref{ex:end}. Based on our experiments, this method works quite well. 

We write \eqref{eq:autosystem} as $U'=F(U)$, where $U=(\bar u,w,z)$. The desired solution is a heteroclinic orbit in the phase space $\RR^3$ which connects the equilibria $U_\spm=(u_\spm,0,0)$.
The linearization is straightforward; we see immediately that
\beq
\dif F(U)=
\begin{pmatrix}
a^1(\bar u) & 0 & 0 \\
\frac{\pd^2 f^1}{\pd \bar u^2} & a^1(\bar u) & 0 \\
\frac{\pd^2 f^1}{\pd\bar u^2}+\tau_\circ+\xi_\circ a^2(\bar u) & 0 & a^1(\bar u)
\end{pmatrix}\,.
\eeq
Evidently, the eigenvalues of $\dif F(U_\spm)$ are, with multiplicity three, simply $a^1(u_\spm)$. From \eqref{eq:lax} this gives a connection from a stable node to an unstable node in $\RR^3$. We expect then to introduce two parameters to completely parametrize the solutions; as noted in Remark~\ref{rem:ab}, an orthogonality condition will select a particular solution for the computation of $\beta$.

\subsubsection{Numerical Implementation}
We truncate the problem to the computational domain $[-L,L]$. Thus, the system becomes 
\beq
U'=F(U)\,,\quad x\in[-L,L]\,.
\eeq
At this point we expect to incorporate projective boundary conditions at $\pm L$. 
Next, for convenience, we double the variables and rescale the problem to the unit interval. Thus we consider the problem
\[
U_r'=LF(U_r)\,,\quad U_\ell'=-LF(U_\ell)\,,     \quad x\in[0,1]\,,
\]
where $U_r(x)=U(Lx)$ and $U_\ell(x)=U(-Lx)$.
We also implement the classical phase condition to remove the translational invariance from the problem. Folding over the solution makes it simple to include this as a boundary condition at $x=0$. Thus, in this case, we use one phase condition, three matching conditions, and two free parameters to determine the solution. The penultimate step is to generate an approximate solution, or a guess. In this case, we generate our guess by solving the corresponding initial-value problem. In practice, since this is a sink-source connection, it is not too difficult to generate a good guess by this method. Once a guess of sufficient quality is found; solutions for nearby parameter values can be found easily by continuation. Finally, we solve the boundary-value problem using \textsc{MatLab}'s routine \texttt{bvp5c}; this is a code that implements the four-stage Lobatto IIIa formula; this is a collocation formula. The collocation polynomial provides a $C^1$-continuous solution that is fifth-order accurate uniformly in the computational domain $[0,1]$. Sample solutions using this procedure are plotted in Figure \ref{fig:end} and Figure \ref{fig:compare}.

\begin{remark}
A natural question concerns the determination of the smallest size of $L$ that guarantees that the numerical approximation fully resolves the features of the true problem on $\RR$. Here, for the model problem that we consider, we are content to do so in an ad hoc way. Since our ultimate interest is in $\sgn\re\beta$, we merely verify that the value of this quantity is stable as $L$ is increased; we take this as  evidence that the computation is sufficiently resolved. See Table~\ref{tab:beta1}.
\end{remark}

\subsubsection{Examples}


\begin{example}[$f^1(u)=u^2/2$, $f^2(u)=u^2$]\label{ex:usquared}
This is the same as Example \ref{ex:exact}, and so there is an exact solution available. The system takes the form 
\begin{subequations}
\begin{align}
\bar u' & = \frac{1}{2}(\bar u^2-1)\,, \\
w' & = \bar u w\,, \\
v' & = \bar u v + \bar u^2-1\,.
\end{align}
\end{subequations}
The approximate solution computed using the method described in Section \ref{ssec:auto} with $L=20$ is plotted against the exact solution in Figure \ref{fig:compare}. We note that the 2-norm error between the approximate solution and the exact solution on their common domain is well-controlled by the built-in error control features of \texttt{bvp5c}. 
\end{example}

\begin{example}[Moving end state]\label{ex:end}
In this example, we take as before 
\(
f^1(u)=u^2/2,
\)
and we  take 
\[
f^2(u)=\sin(4\pi u)
\]
for the transverse flux $f^2$. Our aim in this example is to mimic the kind of calculation that one would do in practice, searching for the point in parameter space (e.g., the value of $u_\sm$) at which $\sgn\Re\beta$ changes sign. To that end, we systematically increase the value of $u_\sm$ and recompute the $\bar u$ and $y$ for each new value of the end state. This is precisely the kind of computation that is well suited for continuation. 
The results are plotted in Figure \ref{fig:end}.  
\begin{figure}[ht] 
   \centering
   \begin{tabular}{ccc}
  (a) \includegraphics[width=1.7in]{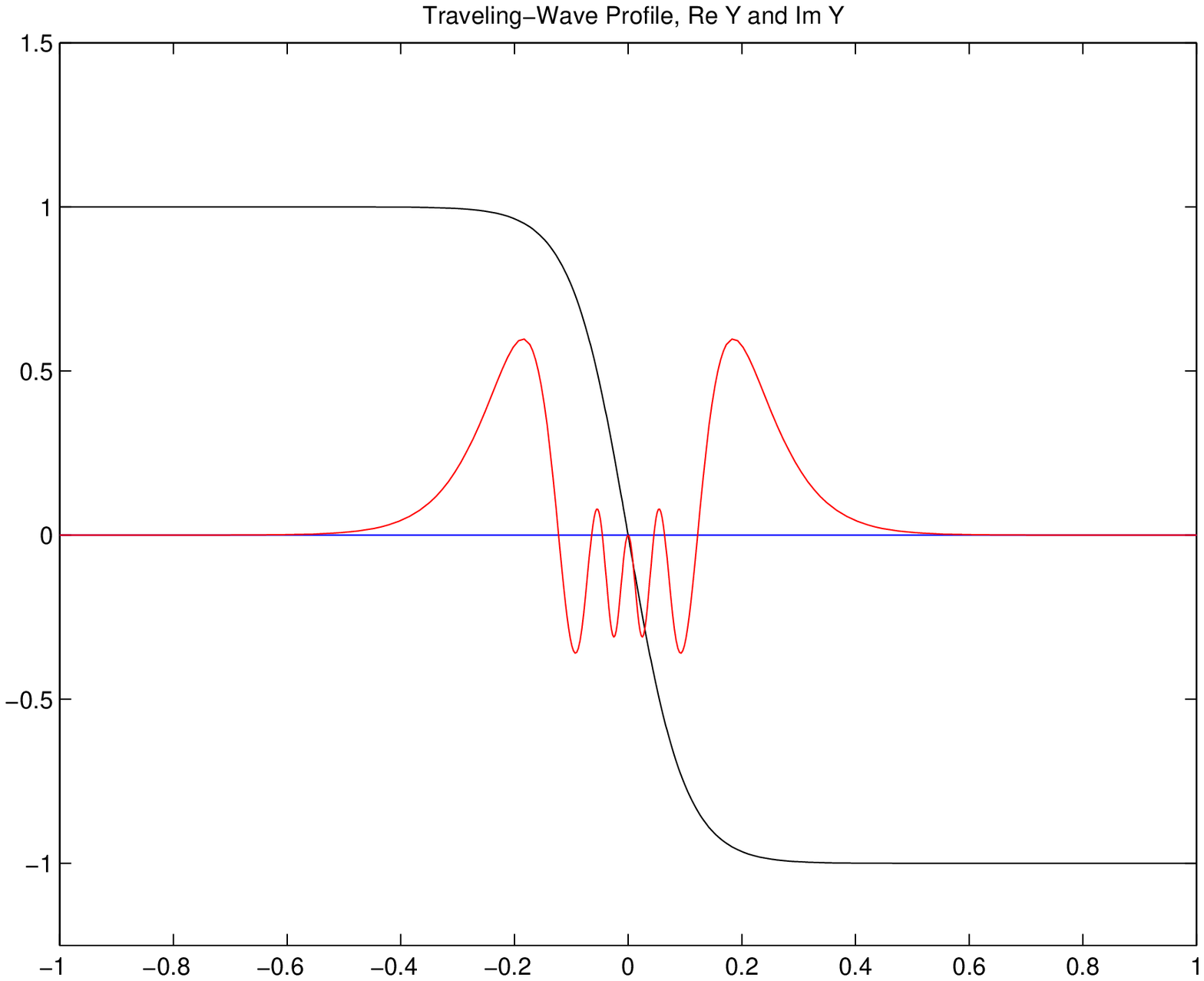} & (b) \includegraphics[width=1.7in]{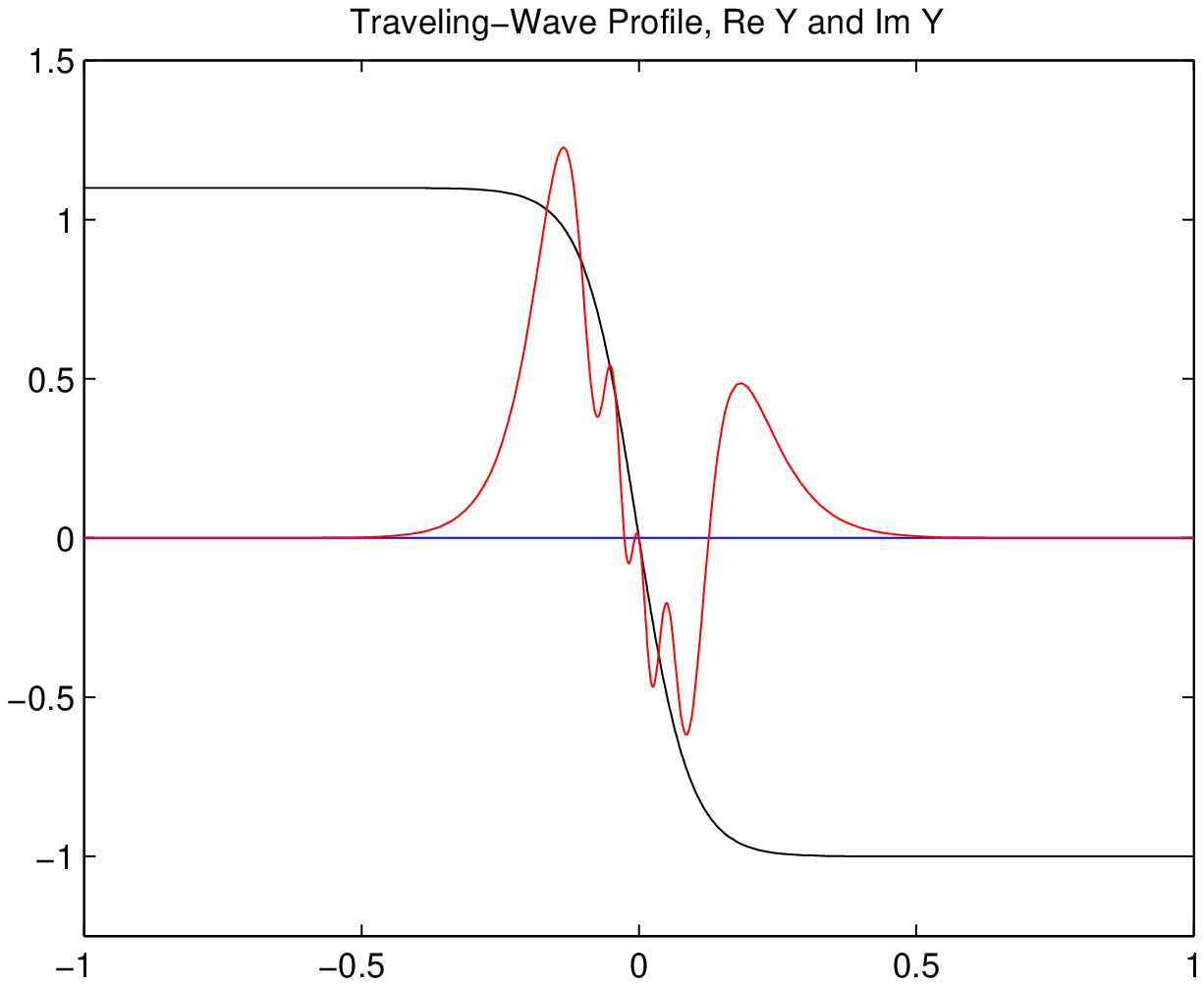} & (c) \includegraphics[width=1.7in]{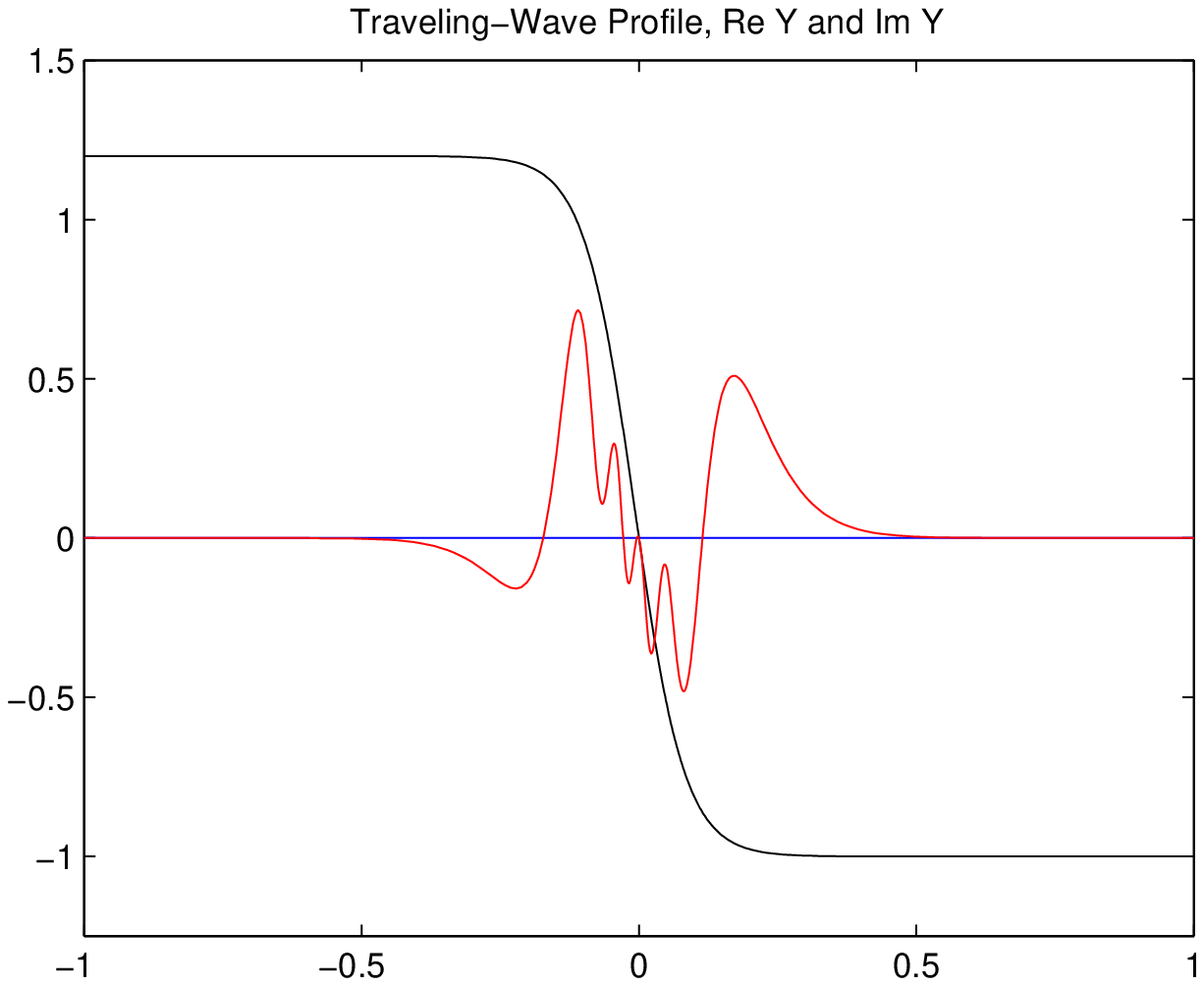} \\
   (d) \includegraphics[width=1.7in]{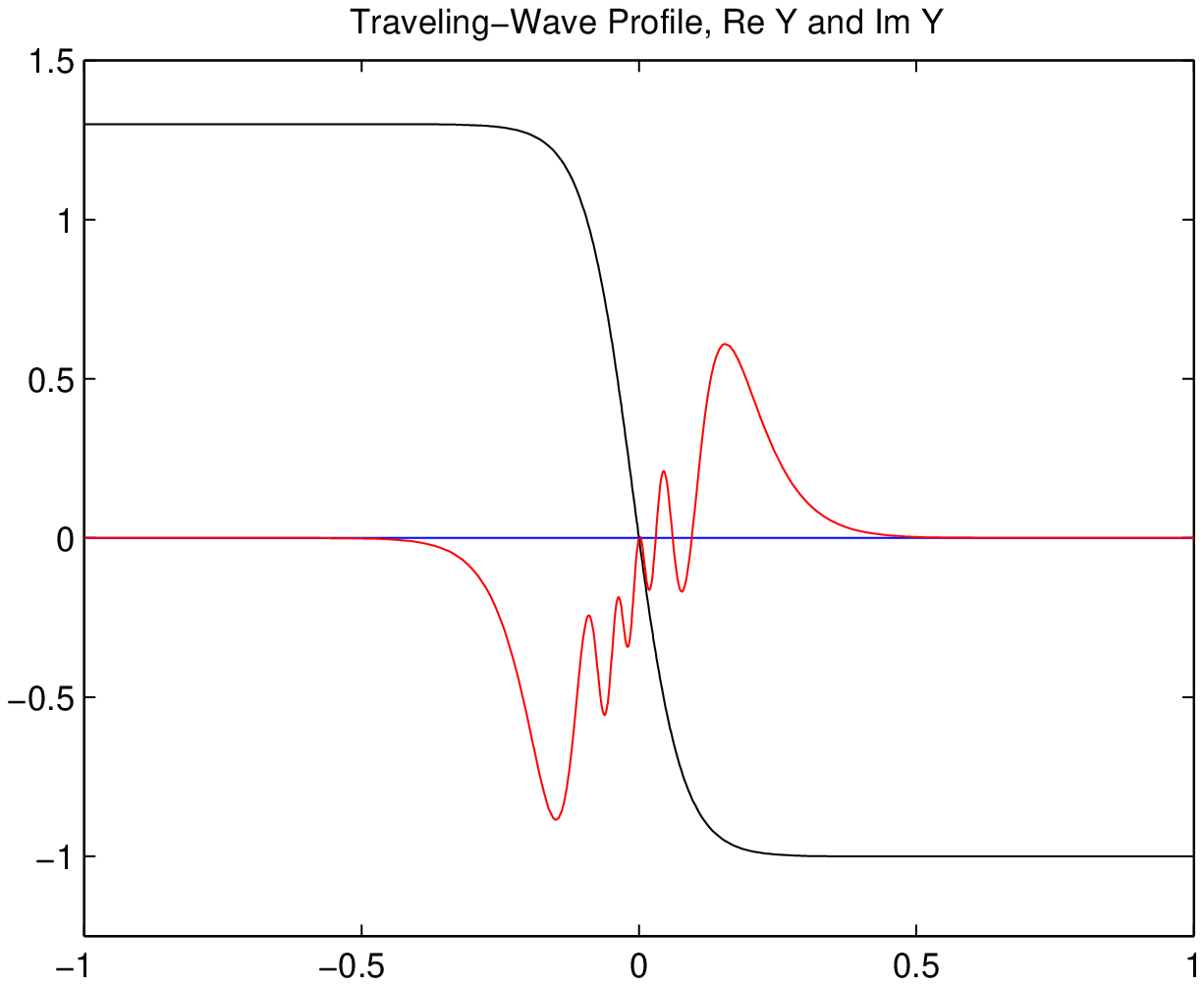} & (e) \includegraphics[width=1.7in]{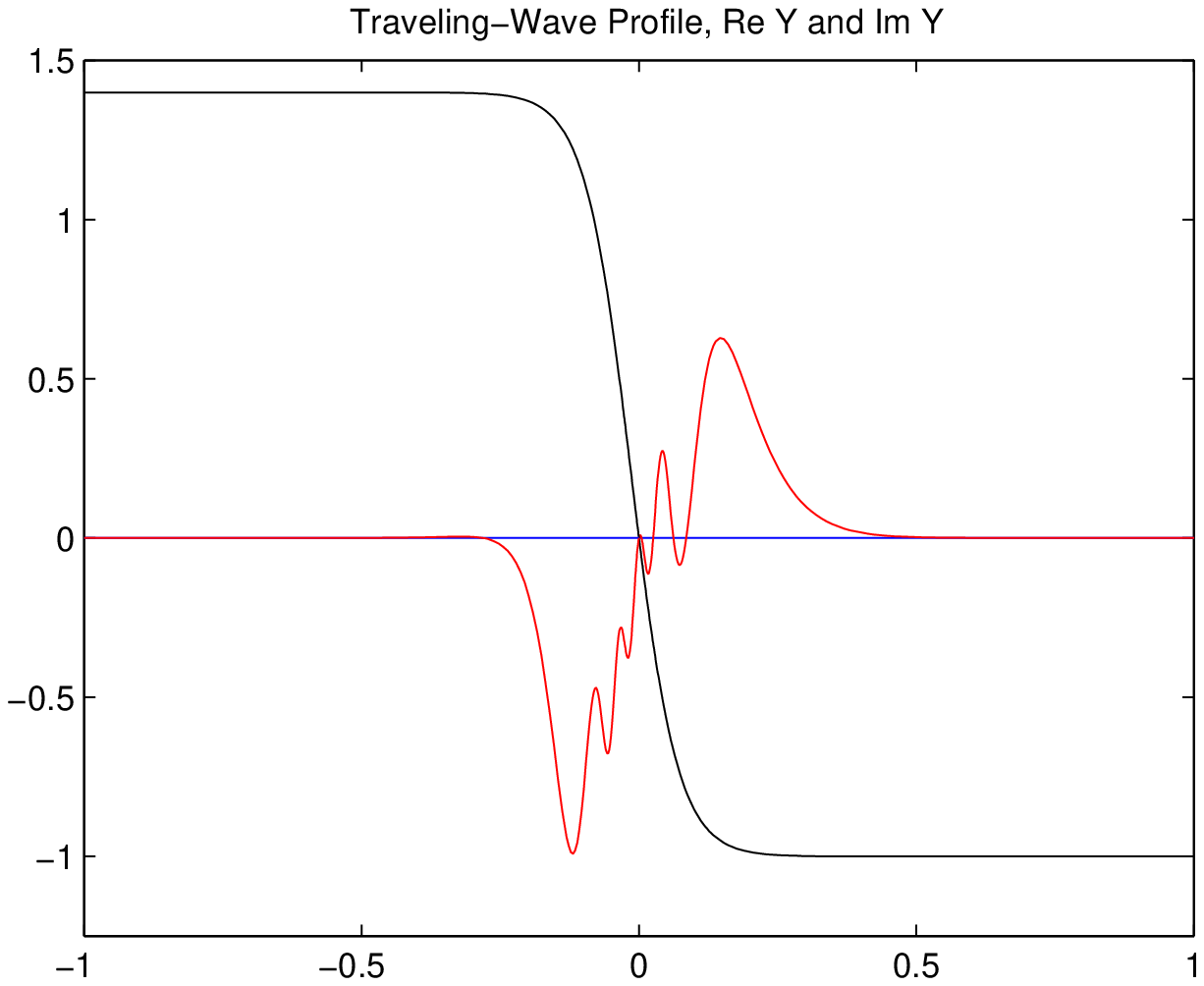} & (f)\includegraphics[width=1.7in]{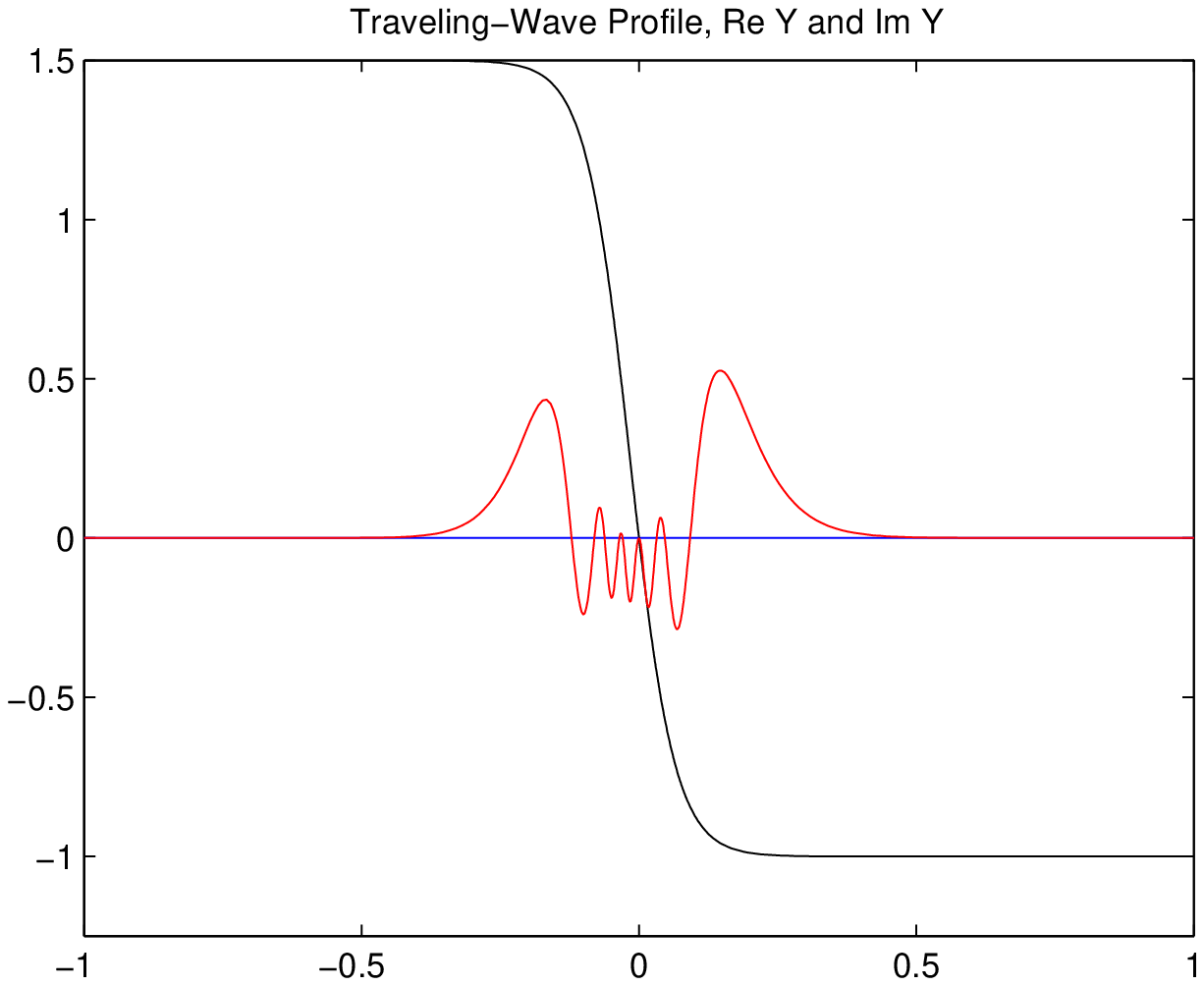}
   \end{tabular} 
   \caption{The effect of moving the left state. Black: profile $\bar u$, Blue: $w=\re \tilde y$, Red: $v=\im \tilde y$: (a) $u_\sm=1$. (b) $u_\sm=1.1$. (c) $u_\sm=1.2$. (d) $u_\sm=1.3$. (e) $u_\sm=1.4$. (f) $u_\sm=1.5$.}
   \label{fig:end}
\end{figure}
\end{example}

\subsection{Comparison of two methods}
In the context of the scalar conservation law \eqref{eq:law}, both methods work well. Our experiments show that some care should be exercised in approximating $\tilde y$ via the integrating factor method; in particular, the overall quality of the computation depends on the approximation of the integrals in \eqref{eq:v}. Our implementation uses Simpson's rule to approximate these integrals. It is worth noting that the profile $\bar u$ is computed independently in this first method; for example, the calculation does not take any account of the transverse flux $f^2$. By way of comparison, error control in the coupled formulation is ``automatic'' since we use built-in convergence tolerances in the package \texttt{bvp5c} to control the quality of the approximation across the entire computational domain; That is, $\bar u$ and $y$ are treated on the same footing, and this method takes account of the entire structure of the problem at each stage of the iteration. In Figure \ref{fig:compare} we plot the exact solution from Example \ref{ex:exact} against the approximate solutions obtained by the two methods. The caption of that figure records the error between the approximate solutions and the exact solution in the 2-norm, given by  
\beq\label{eq:norm}
\norm{x}{2}=\sqrt{\sum_{j=1}^Nx_j^2}\,.
\eeq

On the other hand, a distinction of note between the methods is in the complexity of the nonlinear two-point boundary-value problem that must be solved. Our technique for solving these problems by collocation hinges on finding or computing a suitable initial guess. This guess is used as the seed in an iterative solution of the nonlinear equations for the coefficients of the collocation polynomial. In the first method, based on the integrating factor, a relatively simple boundary-value problem needs to be solved; one expects that it is, generally, much easier to generate a good initial guess for this problem than to find a similarly good initial guess for the coupled formulation of the problem.
\begin{figure}[ht] 
   \centering
   \includegraphics[width=5.5in]{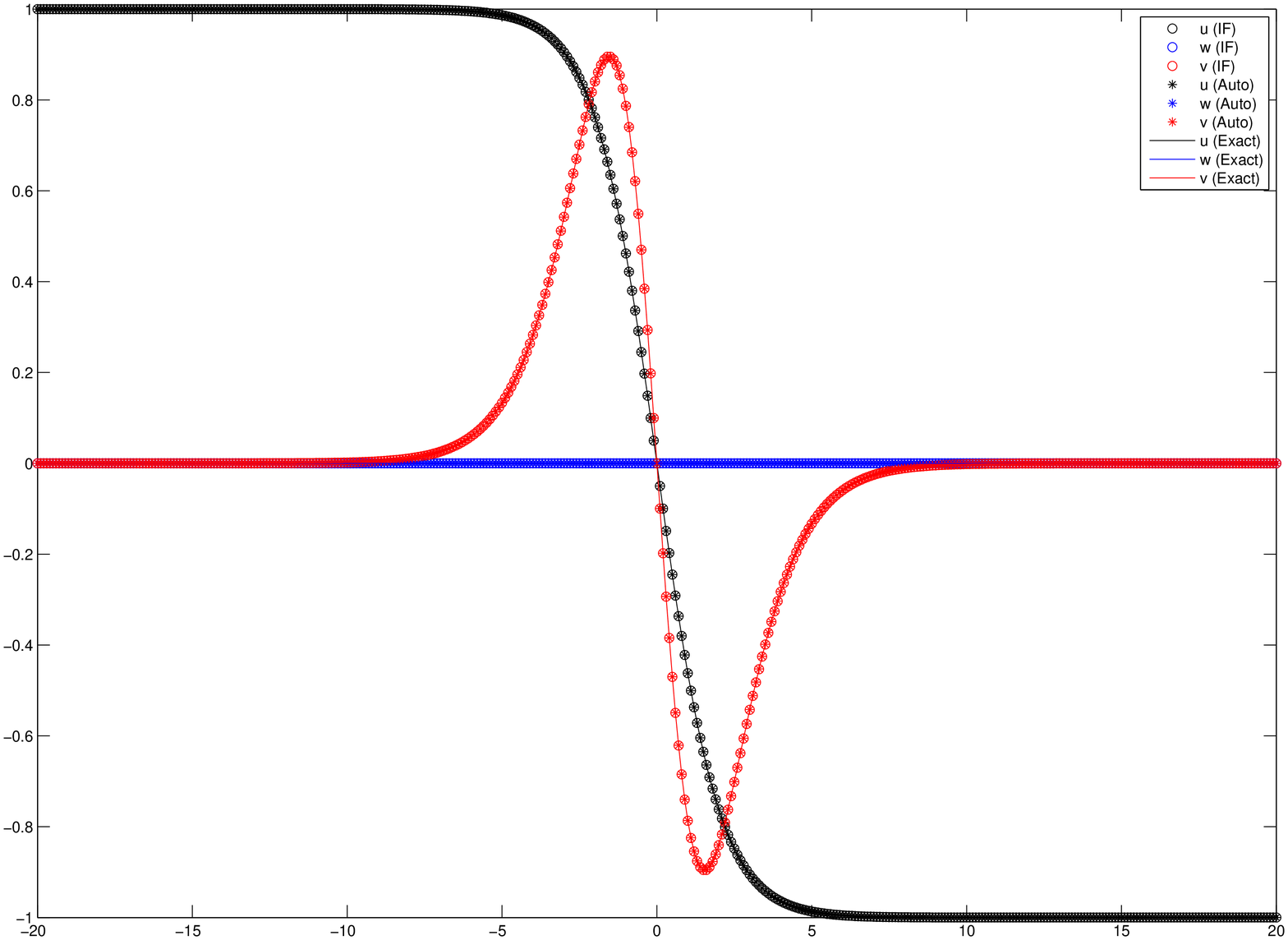} 
   \caption{The computed solutions (autonomous = stars, integrating factor = circles) plotted against the exact solution (solid lines) from Example \ref{ex:exact}. Here, $L=20$, and  $\norm{\bar u_\mathrm{auto}-\bar u_\mathrm{exact}}{2}$=1.0470e-07, $\norm{w_\mathrm{auto}-w_\mathrm{exact}}{2}$=0, and $\norm{ v_\mathrm{auto}- v_\mathrm{exact}}{2}$= 4.42128e-07. Similarly, $\norm{\bar u_\mathrm{int}-\bar u_\mathrm{exact}}{2}$=1.45990e-06, $\norm{w_\mathrm{int}-w_\mathrm{exact}}{2}$=0, and $\norm{ v_\mathrm{int}- v_\mathrm{exact}}{2}$= 2.79917e-04. The error in the autonomous formulation is controlled directly by the convergence criteria of \texttt{bvp5c} while the error in the IF formulation depends on the accuracy of the computed profile $\bar u$ and the quality of the approximation of the integrals in \eqref{eq:v}.
}
\label{fig:compare}
\end{figure}

\section{Conclusion}\label{sec:conclude}
\subsection{Calculating $\beta$}\label{ssec:betacalc}
Finally, with our approximations of $\bar u$ and $\tilde y$ in hand, we proceed to compute $\beta$. Examining \eqref{eq:beta1}, we observe that the formula for $\beta$ may be rewritten as 
\begin{align}
\beta &  =\frac{\Gamma\mathcal{I}}{\Gamma\Delta_\lambda}\nonumber\\
	&= \jump{u}^{-1}\int_{-\infty}^\infty 2(\mi\tau_\circ +\mi\xi_\circ a^2(\bar u))(w+\mi v)+2\xi_\circ^2\bar u'\,\dif x \nonumber\\
	&\approx \frac{2}{\jump{u}}\int_{-L}^L(\mi\tau_\circ +\mi\xi_\circ a^2(\bar u))(w+\mi v)+\xi_\circ^2(f^1(\bar u)-f^1(u_\sm)-s(\bar u-u_\sm))\,\dif x\,. \label{eq:beta-approx}
\end{align}
Thus, from \S\ref{ssec:if} and \S\ref{ssec:auto} we have two distinct ways of obtaining approximations for $\bar u$, $w$, and $v$ on the computational domain $[-L,L]$ in order to approximate the integral in \eqref{eq:beta-approx}. In Table \ref{tab:beta1} we compare the values of $\beta$ computed by each of the methods for the problem in Example \ref{ex:exact}. In each case, we approximate the integral in \eqref{eq:beta-approx} by the trapezoid rule, and we denote the different values obtained by $\beta_\mathrm{auto}$ and $\beta_\mathrm{int}$. 

\begin{table}
\begin{tabular}{l|ccc}
$L$ & 10 & 20 & 30 \\ \hline 
$\beta_\mathrm{auto}$ &  $9.9918+0.0000\mi$ & $10.0000+0.0000\mi$ & $10.0000+0.0000\mi$ \\
$\beta_\mathrm{int}$ & $9.9919+0.0000\mi$ & $10.0001+0.0000\mi$& $10.0001+0.0000\mi$ 
\end{tabular}
\caption{The computation of $\beta$ for Example \ref{ex:exact}: $\beta_\mathrm{exact}=10$.  }
\label{tab:beta1}
\end{table}

\subsection{Discussion}\label{ssec:discuss}
We have considered Zumbrun \& Serre's refined stability condition \cite{ZS_IUMJ99} in the simplest possible setting, and we have proposed and implemented two methods for practically computing the condition. Both methods work well in the present setting, and we believe that the approaches for computing $\tilde y$ from \S\ref{sec:compute} are potentially useful for interesting physical versions of this problem. Indeed, our principal interest is to use the model problem here as a stepping stone towards the analogous problem for a physical system ($n>1$) which possesses open set of weakly stable waves. For isentropic gas dynamics in two space dimensions \eqref{eq:gas}, an example of such a family has been given by Majda \cite{M_book}. In that case $n=3$, and the fundamental new issue that arises is the presence of \emph{slow modes}, see \cites{B-GSZ_SIAMJMA01,B-GSZ_ZAA08}. For example, in the case of a Lax $1$-shock, there are two outgoing characteristic directions, and this alters the calculation from the very beginning. For example, the computation of $\Delta$ is more involved (although it is still computable in closed form \cite{S_book}). In this case the Lopatinski\u\i\ determinant takes the form  
\beq\label{eq:gas_lop}
\Delta(\lambda,\xi)=\det(\vec{r}_2^\sp(\lambda,\xi),\vec{r}_3^\sp(\lambda,\xi),\lambda\jump{U}+\mi\xi\jump{f^2(U)})\,.
\eeq
where the vectors $\{\vec{r}_2^\sp(\lambda,\xi),\vec{r}_3^\sp(\lambda,\xi)\}$ are a basis for the unstable subspace of $\mathcal{A}_\sp$ where
\beq\label{eq:calAgas}
\mathcal{A}_\spm(\lambda,\xi)=\Big(\lambda \dif F^0(U_\spm)+\mi\xi\dif F^2(U_\spm)\Big)(\dif F^1(U_\spm))^{-1}\,,
\eeq
with $U=(\rho,u_1,u_2)^\mathrm{t}$ 
and
\beq\label{eq:flux}
F^0(U)=\begin{pmatrix} \rho \\ \rho u_1 \\ \rho u_2\end{pmatrix}\,,\quad
F^1(U)=\begin{pmatrix} \rho u_1 \\ \rho u_1^2+p(\rho) \\ \rho u_1 u_2\end{pmatrix}\,,\quad
F^2(U)=\begin{pmatrix} \rho u_2 \\ \rho u_1u_2 \\ \rho u_2^2 +p(\rho)\end{pmatrix}\,.
\eeq
In the fluxes $F^j$, the form of the pressure $p$ is prescribed as the equation of state---the constitutive relation that specifies the nature of the gas. Mathematical treatments of gas dynamics frequently take a ``$\gamma$-law'' gas\footnote{Of course, the parameter $\gamma$ appearing in \eqref{eq:gamma-law} is a physical constant and is not related to $\gamma=\Re\lam$ used throughout this paper. Similarly, the gas density $\rho$ appearing in \eqref{eq:gas} and \eqref{eq:flux} is not to be confused with the radial coordinate $\rho$ used in the Evans-function calculations in \S\ref{sec:evans}.}
\beq\label{eq:gamma-law}
p(\rho)=a_0\rho^\gamma\,,
\eeq
where $a_0$ is a positive constant and $\gamma>1$. For such a pressure law, there are no weakly stable shocks \cite{S_book}. However, thermodynamically admissible perturbations of a $\gamma$-law pressure function can open up regions of weak stability \cites{ABLLMSX,M_book}.
Additionally, the presence of slow modes introduces an additional boundary term $\mathcal{B}$, as described in \cite{B-GSZ_ZAA08}, into the formulation of $\beta$. The term $\mathcal{B}$ is associated with the absence of a spectral gap for the linearized operator $\mathscr{L}(0)$, and $\mathcal{B}$ is constructed from the right and left eigenvectors of $\mathcal{A}_\spm$ in \eqref{eq:calAgas}; in the case of \eqref{eq:gas}, these are known explicitly. Finally, in the physical cases, low-frequency behavior on $\pd H$ can be complicated due to the presence of nonempty \emph{glancing} and \emph{variable multiplicity} sets $\mathscr{G}$ and $\mathscr{V}$. In the case of \eqref{eq:gas} $\mathscr{V}$ is empty and $\mathscr{G}$ can be computed explicitly. The application of the ideas in \S\ref{sec:compute} to this problem is a topic of our current investigation.

\section*{Acknowledgement}
Research of all authors was supported in part by the National Science Foundation under grant number DMS-0845127.   
The authors would also like to thank Mark Ablowitz and Harvey Segur for a helpful suggestion.

\end{document}